\documentclass{amsart}
\usepackage{amsfonts}
\usepackage{amsmath}
\usepackage{amssymb}
\usepackage{amsthm}
\usepackage{esint}
\usepackage{enumerate}
\usepackage{cite}

\newcommand{\sps}[1]{\left( #1 \right)}

\newcommand{\eval}[1]{\left. #1 \right\rvert}
\ifdefined\stretchybydefault
\newcommand{\abs}[1]{\left \lvert #1 \right \rvert}

\newcommand{\norm}[1]{\left \lVert #1 \right \rVert}
\newcommand{\jap}[1]{\left\langle #1 \right\rangle}

\else

\newcommand{\abs}[1]{\lvert #1 \rvert}

\newcommand{\norm}[1]{\lVert #1 \rVert}
\newcommand{\jap}[1]{\langle #1 \rangle}
\fi
\newcommand{\loc}{{\mathrm{loc}}}

\DeclareMathOperator{\Lip}{Lip}

\DeclareMathOperator{\divergence}{div}
\renewcommand{\div}{\divergence}

\newcommand{\R}{\mathbb{R}}

\newcommand{\C}{\mathbb{C}}

\newcommand{\comment}[1]{}

\renewcommand{\bar}{\overline}

\let\lll\ll
\let\ggg\gg
\renewcommand{\ll}{\lesssim}
\renewcommand{\gg}{\gtrsim}

\newcommand{\pd}{\partial}
\newcommand{\thnorm}[1]{|\kern-1pt|\kern-1pt| #1 |\kern-1pt|\kern-1pt|}
\theoremstyle{remark}
\newtheorem*{remark*}{Remark}
\theoremstyle{definition}
\newtheorem*{dfn*}{Definition}

\theoremstyle{plain}
\newtheorem*{theorem*}{Theorem}
\newtheorem*{proposition*}{Proposition}

\newtheorem{theorem}{Theorem}[section]
\newtheorem*{lemma*}{Lemma}
\newtheorem{lemma}{Lemma}[section]
\newtheorem*{corollary*}{Corollary}
\newtheorem{corollary}{Corollary}[section]

\begin{document}

\title{Uniqueness in Calder\'on's problem with Lipschitz conductivities}
\author{Boaz Haberman and Daniel Tataru}
\thanks{Both authors were partially supported by NSF grant DMS0801261.}
\begin{abstract}
  We use $X^{s,b}$-inspired spaces to prove a uniqueness result for
  Calder\'on's problem in a Lipschitz domain $\Omega$ under the
  assumption that the conductivity lies in the space
  $W^{1,\infty}(\bar \Omega)$. For Lipschitz conductivities, we obtain
  uniqueness for conductivities close to the identity in a suitable
  sense. We also prove uniqueness for arbitrary $C^1$ conductivities.
\end{abstract}
\maketitle

\section{Introduction}
Calder\'on's problem asks whether one can recover the conductivity of
an object in its interior based on measurements made at the
boundary. Let $\Omega\subset \R^d$ be some bounded domain with
Lipschitz boundary, and let $\gamma$ be a strictly positive
real-valued function defined on $\Omega$ which gives the conductivity
at a given point. An electrical potential $u$ in this situation
satisfies the conductivity equation
\[L_\gamma u = 0,\]
where
\[L_\gamma u := \div(\gamma \nabla u).\]
Given $f\in H^{1/2}(\pd \Omega)$, there exists a unique solution $u_f$ to the Dirichlet problem 
\begin{align*}
L_\gamma u_f & = 0 \quad \text{in $\Omega$}\\
\eval{u_f}_{\pd\Omega} & = f,
\end{align*}
and hence we may formally define the Dirichlet-to-Neumann map $\Lambda_\gamma$ by
\[
\Lambda_\gamma(f) := \eval{\gamma \frac{\pd u_f}{\pd \nu}}_{\pd \Omega},
\]
where $\pd/\pd \nu$ is the outward normal derivative at the
boundary. If $\gamma \in \Lip(\bar \Omega)$, then $\Lambda_\gamma$ is
a well defined map from $H^{1/2}(\pd \Omega)$ to $H^{-1/2}(\pd
\Omega)$. In physical terms, this map encodes how the boundary
potential determines the current flux across the
boundary. Calder\'on's inverse problem is to reconstruct $\gamma$ from
the map $\Lambda_\gamma$; an obvious condition for this to be possible
is that the map $\gamma \mapsto \Lambda_\gamma$ be injective. 

A key result in this direction was obtained by Sylvester and Uhlmann
in~\cite{Sylvester1986}; there they proved uniqueness for $C^2$
conductivities. Later Brown~\cite{Brown1996} relaxed the regularity of
the conductivity to $3/2+\epsilon$ derivatives.  This was followed by
uniqueness for $W^{3/2,\infty}$ conductivities in~\cite{Paivarinta2003}
and for $W^{3/2,p}$ (with $p > 2n$) in~\cite{Brown2003}.  

It has
been conjectured by Uhlmann that the optimal assumption is that the conductivities are Lipschitz. Our main theorem asserts that
uniqueness holds for $C^1$ conductivities and Lipschitz conductivities close to the identity. We have no
counterexample to prove that this result is optimal.

\begin{theorem}
Let $\Omega \subset \R^d$, $d\geq 3$ be a bounded domain with Lipschitz boundary. For $i = 1,2$, let $\gamma_i \in W^{1,\infty}(\bar \Omega)$ be real valued functions, and assume there is some $c$ such that $\gamma_i > c > 0$. Then there exists a constant $\epsilon_{d,\Omega}$ such that if each $\gamma_i$ satisfies either $\norm{\nabla \log \gamma_i}_{L^\infty(\bar \Omega)} \leq \epsilon_{d,\Omega}$ or $\gamma_i \in C^1(\bar \Omega)$ then $\Lambda_{\gamma_1} = \Lambda_{\gamma_2}$ implies $\gamma_1 = \gamma_2$.
\end{theorem}

The basic approach to this problem in this paper is the method
introduced by Sylvester and Uhlmann in~\cite{Sylvester1986} based on
the ideas in~\cite{CalderonAlbertoP.}. Kohn and
Vogelius~\cite{Kohn1984} showed that for smooth conductivities, the
map $\gamma \to \Lambda_\gamma$ determines the values of $\gamma$ and
all of its derivatives on $\pd \Omega$. This was improved to Lipschitz
conductivities in domains with Lipschitz boundary by Alessandrini
in~\cite{Alessandrini1990}. Using this result, we may reduce the
inverse problem for the conductivity equation to an inverse problem
for the Schr\"odinger equation $(-\Delta + q)v=0$, where the potential
$q$ is defined by $q = \gamma^{-1/2} \Delta \gamma^{1/2}$. If $u$ is a
solution to the equation $L_\gamma u = 0$, then $v=\gamma^{1/2} u$
satisfies $(-\Delta + q)v =0$. The corresponding Dirichlet-to-Neumann
map is defined by
\[\Lambda_q(f) := \eval{\frac{\pd v_f}{\pd \nu}}_{\pd \Omega},\]
where $v_f$ is now a solution to $(-\Delta + q)v = 0$ with boundary
data $f$. For $q$ corresponding to a $C^2$ conductivity, at least,
this map is well-defined. If $\gamma_1$ and $\gamma_2$ satisfy
$\Lambda_{\gamma_1}=\Lambda_{\gamma_2}$, then by the boundary
identification result we have $\Lambda_{q_1}=\Lambda_{q_2}$ for $q_i =
\gamma_i^{-1/2}\Delta \gamma_i^{1/2}$, where $\Lambda_q$ is the
Dirichlet-to-Neumann map $f \mapsto \pd v_f/\pd \nu$ corresponding to
the equation $(-\Delta + q)v = 0$. Now, if $u$ satisfies $(-\Delta +
q)u = 0$ inside $\Omega$ in the weak sense, then for any $v$ we have
\begin{align*}
0&=\int_{\Omega} (-\Delta + q)u v dx \\
& = \int_{\Omega} (\nabla u \cdot \nabla v + quv) dx - \int_{\pd \Omega} \frac{\pd u}{\pd \nu} v d\sigma.
\end{align*}
Hence if $\eval{u}_{\pd \Omega} = f$ and $\eval{v}_{\pd \Omega} = g$, we have
\begin{align*}
(\Lambda_q f,g)_{L^2(\pd \Omega)} & = \int_{\pd \Omega} \frac{\pd u}{\pd \nu} v d\sigma \\
& = \int_{\Omega} (\nabla u \cdot \nabla v + quv) dx.
\end{align*}
In particular, if $\Lambda_{q_1}=\Lambda_{q_2}$ and $(-\Delta + q_i)u_i = 0$, then a simple calculation shows that
\[\int_{\Omega} (q_1-q_2) u_1u_2 dx =0.\]
Using boundary determination, it is possible to extend the $\gamma_i$ to functions in $W^{1,\infty}(\R^d)$ (or $C^1$) so that $\gamma_1 = \gamma_2$ on $\R^d - \Omega$. Given this, we can extend the domain of integration to all of $\R^d$, and obtain
\[\int_{\R^d} (q_1-q_2) u_1u_2 dx = 0.\]
From this discussion, it follows that one way to show that the potentials $q_1$ and $q_2$ coincide is to produce enough solutions to the corresponding Schr\"odinger equations that their products are dense in some sense. This idea goes back to the original paper of Calder\'on~\cite{CalderonAlbertoP.}. In~\cite{Sylvester1987}, Sylvester and Uhlmann proved a uniqueness result for $C^2$ conductivities by constructing complex geometrical optics solutions of the form $u_i = e^{x\cdot \zeta_i}(1 + \psi_i)$. Here the $\zeta_i\in \C^d$ are chosen so that $\zeta_i \cdot \zeta_i = 0$, so that $e^{x\cdot \zeta_i}$ is harmonic, and $e^{x\cdot \zeta_1} e^{x\cdot \zeta_2} = e^{ix\cdot k}$ for some fixed frequency $k\in \R^d$. In three or more dimensions, these conditions give sufficient freedom that is possible to choose an infinite family of pairs $\zeta_1,\zeta_2$ with $\abs{\zeta_i} \to \infty$. This in turn ensures that the remainders $\psi_i$ decay to zero in some sense as $\abs{\zeta_i} \to \infty$, so that the product $u_1 u_2$ converges to $e^{ix\cdot k}$. Uniqueness then follows from Fourier inversion.

To construct these CGO solutions, fix $\zeta \in \C^d$ such that $\zeta \cdot \zeta = 0$, and note that $e^{-x\cdot \zeta} \Delta( e^{x\cdot \zeta}\psi) = (\Delta + 2\zeta \cdot \nabla) \psi$. Thus $u = e^{x\cdot \zeta}(1+\psi)$ solves $\Delta u = qu$ if
\begin{equation}
\label{cgoequation}
\Delta_\zeta \psi := \Delta \psi + 2 \zeta \cdot \nabla \psi = q (1+\psi).
\end{equation}
Let $m_q$ be the map sending $\psi$ to $q\psi$. We will treat this equation perturbatively, by viewing $\Delta_\zeta - m_q$ as a pertubation of $\Delta_\zeta$. The operator $\Delta_\zeta$ has a right inverse defined by 
\[\widehat{\Delta_\zeta^{-1} f}(\xi) = p_\zeta(\xi)^{-1} \hat f(\xi),\]
where
\[p_\zeta(\xi) = -\abs{\xi}^2 + 2i\zeta \cdot \xi.\]
To construct a solution to~\eqref{cgoequation} using a fixed-point argument, we need to bound the operators $\Delta_\zeta^{-1}$ and $m_q$ in some iteration spaces. Sylvester and Uhlmann~\cite{Sylvester1987} showed that
\begin{equation}
\label{inferiorestimate}
\norm{\Delta^{-1}_\zeta}_{L^2_{\delta+1}\to L^2_\delta} \ll \abs{\zeta}^{-1},
\end{equation}
where $-1<\delta<0$ and $\norm{u}_{L^2_\delta} = \norm{\jap{x}^\delta u}_{L^2}$. For $\gamma \in C^2$, we have $q \in L^\infty$, and the bound $\norm{m_q}_{L^2_\delta \to L^2_{\delta + 1}} \ll_q 1$ is trivial. Combining these two bounds closes the iteration argument, showing that the CGO solutions exist and that the remainder $\psi$ goes to zero in some suitable sence as $\abs{\zeta} \to \infty$. If $\gamma$ does not have two derivatives, then it is possible to salvage this argument by viewing $q$ as having negative regularity. Brown~\cite{Brown1996} used the estimate~\eqref{inferiorestimate} to derive a bound for $\Delta^{-1}_\zeta$ on certain weighted Besov spaces of negative order. Combined with a corresponding bound for $m_q$, this gives uniqueness under the assumption of $3/2 + \epsilon$ derivatives. The general outline and much of the notation in this paper will follow Brown, and the main focus will be on improving the estimates. Uniqueness for conductivities is $W^{3/2,\infty}$ was shown in~\cite{Paivarinta2003} using a much more involved approximation argument; this result was later improved to $W^{3/2,p}$ (with $p > 2n$) in~\cite{Brown2003}. We will use a simpler approximation argument to obtain the results for $C^1$ and Lipschitz conductivities. At this regularity there are only partial results in the literature; for example see~\cite{Greenleaf2003}, which establishes global uniqueness for certain conductivities in $C^{1+\epsilon}$, and~\cite{Kim2008}, which establishes global uniqueness for Lipschitz conductivities that are piecewise smooth across polyhedral boundaries. 

The first main idea in this paper is to use an iteration space that is adapted to the structure of the equation~\eqref{cgoequation}. In the spirit of Bourgain's $X^{s,b}$ spaces~\cite{Bourgain1993}, we define spaces $\dot X^b_\zeta$ by the norm
\[\norm{u}_{\dot X^b_\zeta} = \norm{\abs{p_\zeta(\xi)}^{b} \hat u(\xi)}_{L^2},\]
where $p_\zeta(\xi) = -\abs{\xi}^2 + 2i\zeta \cdot \xi$ is the symbol of $\Delta_\zeta$. In our analysis, we will only need the spaces $\dot X^{1/2}_\zeta$ and $\dot X^{-1/2}_\zeta$. It is easy to see that $\norm{\Delta^{-1}_\zeta}_{\dot X_\zeta^{-1/2} \to \dot X_\zeta^{1/2}} =1$. We will also make use of the inhomogeneous spaces $X_\zeta^{b}$ with norm
\[\norm{u}_{X_\zeta^{b}} = \norm{(\abs{\zeta} + \abs{p_\zeta(\xi)})^{b} \hat u(\xi)}_{L^2}.\]
For a function $\gamma \in W^{1,\infty}(\R^d)$ that is constant outside a compact set, we set $g = \gamma^{1/2}$, and formally define the associated potential $q = g^{-1} \Delta g$. Following~\cite{Brown1996}, we define the ``multiplication by $q$'' map by duality, i.e.
\[\jap{m_q(u),v} = -\int \nabla g \cdot \nabla (g^{-1} u v) dx.\]
We will use the fact that $q$ is compactly supported in an essential
way. It is clear that if $\phi_B$ is a smooth compactly supported
function with $\phi_B = 1$ on a ball containing the support of $q$,
then $\jap{m_q(u),v} = \jap{m_q(u_B),v_B}$, where $u_B = \phi_B u$ and
$v_B = \phi_B v$. By the uncertainty principle, multiplication by the
cutoff $\phi_B$ should smooth things out on the unit scale in Fourier
space. Heuristically, this means that the growth at infinity will be
unchanged, but concentrations of mass (for example near the
characteristic set of $\Delta_\zeta$) should be smoothed out. Thus
$\nabla u_B$ should be comparable to $\Delta_\zeta^{1/2} u$ at
sufficiently high frequencies, since $\abs{p_{\zeta}(\xi)}^{1/2} \sim
\abs{\xi}$ when $\abs{\xi} \ggg \abs{\zeta}$. At low frequencies,
$p_\zeta(\xi) \sim \abs{\zeta}d(\xi,\Sigma_\zeta)$, where
$\Sigma_\zeta$ is the zero set of $p_\zeta$. When we smooth things out
on the unit scale, we have the heuristic $p_\zeta(\xi) \gg
\abs{\zeta}$, so $u_B$ should be bounded by
$\abs{\zeta}^{-1}\Delta_{\zeta}^{1/2} u$. By these types of
considerations, we will obtain the bound
\[\norm{m_q}_{\dot X^{1/2}_\zeta \to \dot X_\zeta^{-1/2}} \ll_{\gamma} 1,\]
where the implied constant is small for large $\zeta$ as long as
$\norm{\log \gamma}_{W^{1,\infty}}$ is small or $\gamma \in C^1$. This
bound will come into play in two ways. First, we will use it to make
the iteration argument work and produce CGO solutions with remainder
terms $\psi_i$ decaying in $\dot X^{1/2}_{\zeta_i}$. Second, a
bilinear version of this bound will help establish that $\int
qe^{ix\cdot k} \psi_1\psi_2 dx \to 0$ as $\abs{\zeta_i} \to \infty$.

In order to show that the remainder $\psi$ goes to zero as
$\abs{\zeta}\to \infty$, we need to have $\norm{q}_{\dot
  X^{-1/2}_\zeta} \to 0$ as $\abs{\zeta} \to \infty$. Unfortunately,
if $\gamma$ is merely in $W^{1+\theta,\infty}$, the obvious estimates
only give $\norm{q}_{\dot X^{-1/2}_\zeta} \ll
\abs{\zeta}^{1/2-\theta}$, which would require $\theta \geq 1/2$ in
order to get decay. The second main idea in this paper is to establish
the estimate $\norm{q}_{\dot X^{-1/2}_\zeta} \to 0$ {\em on average}
(where the average is taken over suitable values of $\zeta$), which
suffices for our purposes since we only require a sequence
$\zeta^{(n)}$ growing to infinity for which our estimates hold.

{\bf Acknowledgments:} The authors are grateful to Gunther Uhlmann 
for introducing them to this problem and for many useful discussions.

\section{Localization estimates}
To exploit the fact that the $q_i$ are compactly supported, we will write $m_{q_i}(u)$ as $m_{q_i}(\phi_B u)$, where $\phi_B$ is a Schwartz cutoff function that is equal to one on an open ball $B$ containing the supports of the $q_i$. This leads us to bound the map $u \mapsto \phi_B u$ with respect to various norms.

We now pass to a somewhat more general framework. Let $v$ and $w$ be two weights on $\R^d$. Defining $Tf = \phi * f$, where $\phi$ is a rapidly decreasing function, we would like to find sufficient conditions for $\norm{T}_{L^2_v \to L^2_w}$ to be bounded. This is equivalent to bounding $\norm{Sf}_{L^2\to L^2}$, where $Sf = w(\xi)^{1/2} \phi * (v(\xi)^{-1/2} f)$.

We prove the following lemma:
\begin{lemma}
Let $v$ and $w$ be nonnegative weights defined on $\R^d$. If $\phi$ is a fixed rapidly decreasing function, then
\[\norm{\phi*f}_{L^2_w} \ll_\phi \min\left\{\sup_{\xi} \sqrt{\int J(\xi,\eta) d\eta},\sup_{\eta} \sqrt{\int J(\xi,\eta)d\xi}\right\}\norm{f}_{L^2_v},\]
where 
\[J(\xi,\eta) = \abs{\phi(\xi-\eta)} \frac{w(\xi)}{v(\eta)}.\]
\end{lemma}
\begin{proof}
We can write
\[\norm{Sf}_{L^2}^2 = \int \sps{\int \phi(\xi-\eta) v(\eta)^{-1/2} f(\eta) d\eta}\bar{\sps{\int \phi(\xi-\zeta) v(\zeta)^{-1/2} f(\zeta) d\zeta}} w(\xi) d\xi.\]
Applying the inequality $ab \leq \frac 1 2 (a^2 + b^2)$ we have
\[\norm{Sf}_{L^2}^2 \leq \iiint \abs{\phi(\xi-\eta) \phi(\xi-\zeta)} v(\eta)^{-1} \abs{f(\eta)}^2 w(\xi) d\eta d\zeta d\xi. \]
Integrating first in $\zeta$, we find that
\[\norm{Sf}_{L^2}^2 \ll \iint \abs{\phi(\xi-\eta)} \frac{w(\xi)}{v(\eta)} \abs{f(\eta)}^2 d\eta d\xi,\]
Equivalently, we may bound the adjoint $S^*$. To describe the adjoint, we compute
\begin{align*}
\jap{Sf,g} & = \int \sps{\int \phi(\xi-\eta) v(\eta)^{-1/2} f(\eta) d\eta} w(\xi)^{1/2} \bar{g(\xi)} d\xi\\
& = \int f(\eta) \bar{\sps{\int \bar{\phi(\xi-\eta)} g(\xi) \frac{ w(\xi)^{1/2}}{ v(\eta)^{1/2}} d\xi}} d\eta\\
& = \jap{f,S^*g},
\end{align*}
so that
\[S^*g(\eta) = \int \bar{\phi(\xi-\eta)} g(\xi) \frac{w(\xi)^{1/2}}{ v(\eta)^{1/2}}d\xi. \]
This means that bounding $T$ from $L_v$ to $L_w$ is the same as bounding $T^*g = \bar{\phi(-\xi)}*g$ from $L_{w^{-1}}^2$ to $L_{v^{-1}}^2$. To do this it suffices to show that
\[\int \abs{\phi(\xi-\eta)}\frac{w(\xi)}{v(\eta)} d\eta = \int J(\xi,\eta) d\eta\]
is uniformly bounded.
\end{proof}

Let $\zeta \in \C^d$ be such that $\zeta \cdot \zeta = 0$. Write $\zeta = s(e_1 - ie_2)$, with $e_1, e_2\in \R^d$ satisfying $e_1 \cdot e_2 = 0$, and define $H$ and $L$ to be projection operators onto high and low frequencies, given by $\widehat{Lu}(\xi) = \chi(\xi/(8s)) \hat u(\xi)$ and $\widehat{Hu}(\xi) = (1-\chi(\xi/(8s)))\hat u(\xi)$, where $\chi$ is a smooth cutoff function supported in $B(0,2)$ for which $\eval{\chi}_{B(0,1)} = 1$. We establish the following estimates for the localizations:
\begin{lemma}
\label{localizationlemma}
Let $\phi_B$ be a fixed Schwartz function, and write $u_B = \phi_B u$. Then the following estimates hold (with constants dependending on $\phi_B$):
\begin{align}
\label{singularity}
\norm{u_B}_{\dot X^{-1/2}_\zeta} &\ll \norm{u}_{X_\zeta^{-1/2}}\\
\label{adjointsingularity}
\norm{u_B}_{X^{1/2}_\zeta} &\ll \norm{u}_{\dot X^{1/2}_\zeta}\\
\label{lfestimate}
\norm{u_B}_{L^2} & \ll s^{-1/2} \norm{u}_{\dot X^{1/2}_\zeta}\\
\label{hfgradestimate}
\norm{\nabla(Hu_B)}_{L^2} &\ll \norm{u}_{\dot X^{1/2}_\zeta}\\
\label{hfestimate}
\norm{Hu_B}_{L^2} & \ll s^{-1} \norm{u}_{\dot X^{1/2}_\zeta},
\end{align}
\end{lemma}
\begin{proof}
We write $\zeta = s(e_1 - ie_2)$, where $e_1,e_2$ are unit length and orthogonal to each other, and extend to a basis $e_1,\dotsc,e_d$ of $\R^d$. We write 
\[p(\xi) = -\abs{\xi}^2 + 2 i \xi \cdot s(e_1 - ie_2) = (s^2-\abs{\xi - se_2}^2) + 2is\xi_1.\]

From this expression, we see that the symbol $p$ vanishes simply on a hypersurface of codimension two, namely the intersection of the sphere $\abs{\xi-se_2} = s$ with the plane $\xi_1 = 0$. Let $\Sigma_\zeta$ denote the set of points where $p(\xi) = 0$, and let $\Sigma_{\zeta,\delta}$ denote the points whose distance from $\Sigma_\zeta$ is at most $\delta$. Here $\delta$ is a small fixed number independent of $s$.

It is clear that $\abs{p(\xi)} \sim \abs{\xi}^2$ for large $\xi$. More precisely, if $\abs{\xi} \geq 8s$, then $\abs{p(\xi)+\abs{\xi}^2} \ll 4s\abs{\xi} \ll \abs{\xi}^2/2$, which implies that $\abs{\xi}^2/2 \leq \abs{p(\xi)} \leq 3\abs{\xi}^2/2$.

On the other hand, if we fix a constant $M$, then for $\abs{\xi} \leq Ms$ we have $\abs{p(\xi)} \sim_M sd(\xi,\Sigma_\zeta)$. It is easy to see that 
\[d(\xi, \Sigma_\zeta) \sim \abs{s - \abs{\xi-se_2}} + \abs{\xi_1}.\]
Hence
\begin{align*}
\abs{p(\xi)} & = \sqrt{((s+\abs{\xi-se_2})(s - \abs{\xi-se_2}))^2 + (2s\abs{\xi_1})^2} \\
& \sim_M s (\abs{s - \abs{\xi-se_2}} + \abs{\xi_1})\\
& \sim s d(\xi,\Sigma_\zeta).
\end{align*}

We are now ready to prove the estimates in the lemma. The main point is to analyze the behavior of the symbol $p$ near the characteristic set $\Sigma_\zeta$ (cf.~\cite{Sylvester1987,Greenleaf2001a} where a similar analysis is carried out). We first establish~\eqref{singularity}, after which the rest of the estimates will follow easily. The estimate~\eqref{singularity} simply reflects the fact that the inhomogeneous $X^{-1/2}_\zeta$ norm is a blurry version of the $\dot X^{-1/2}_\zeta$ norm, where in particular the integrable singularity that arises on the zero set of $p_\zeta$ is smoothed out. To make this precise, we first show that 
\begin{equation}
\label{singbound}
\int \jap{\xi-\eta}^{-M} \frac{1}{d(\xi,\Sigma_\zeta)} d\xi \ll 1.
\end{equation}
This is true for any nice codimension 2 hypersurface $\Sigma_\zeta$, as can easily be seen by using a partition of unity and flattening out the surface. We want to show that this holds independently of $\zeta$, which will be true, roughly speaking, because the surface $\Sigma_\zeta$ only gets flatter as $s \to \infty$. For our purposes it will suffice to treat the case at hand. We split the integral as $\int_{\Sigma_{\zeta,1}} + \int_{\R^d - \Sigma_{\zeta,1}}$. Since $d(\xi,\Sigma_\zeta) \geq 1$ outside $\Sigma_{\zeta,1}$, it is clear that $\int_{\R^d-\Sigma_{\zeta,1}} \ll 1$. It remains to treat the integral over $\Sigma_{\zeta,1}$. Write $\xi = (\xi_1,\xi')$, and pass to polar coordinates in $\xi'$ centered at $(se_2)'$. Then for $\xi = \xi_1e_1 + se_2 + r\omega$ and $\eta = \eta_1e_1 + se_2 + t\nu$, we have $\jap{\xi-\eta}^{-M} \ll \jap{r\omega-t\nu}^{-M} \ll \jap{r\omega - r\nu}^{-M}$ (by the elementary inequality $\abs{r\omega - t \nu} \geq \frac 1 2 \sqrt{r^2+t^2}\abs{\omega-\nu}$.) Also, note that $d(\xi,\Sigma_\zeta) \gg \abs{r-s} + \abs{\xi_1}$, so that 
\[\int_{\Sigma_{\zeta,1}} \ll \int_{S^{n-2}} \int_{-1}^{1} \int_{s-1}^{s+1} \jap{r\omega - r\nu}^{-M} (\abs{r-s} + \abs{\xi_1})^{-1} r^{n-2} dr d\xi_1 d\omega. \]
Since $(\abs{r-s} + \abs{\eta_1})^{-1}$ is integrable with respect to $dr d\eta_1$,
\[\int_{\Sigma_{\zeta,1}} \ll (s+1)^{n-2} \int_{S^{n-2}} \jap{(s-1)\omega - (s-1)\nu}^{-M} d\omega.\]
The quantity $s^{n-2} \int_{S^{n-2}} \jap{s\omega - s\nu}^{-M} d\omega$ is uniformly bounded in $s$, so $\int_{\Sigma_{\zeta,1}} \ll 1$.

In order to prove~\eqref{singularity} and the adjoint estimate~\eqref{adjointsingularity}, we need to show that
\[\int \abs{\phi(\xi-\eta)} \frac{\abs{p(\eta)} + s}{\abs{p(\xi)}} d\xi \ll 1.\]
We split this into two integrals $\int_{\abs{\xi} > 100s}$ and $\int_{\abs{\xi} \leq 100s}$. First we estimate $\int_{\abs{\xi} > 100s}$. Here $\abs{p(\xi)} \sim \abs{\xi}^2$, so 
\[\int_{\abs{\xi}>100s} \ll \int_{\abs{\xi}>100s} \abs{\phi(\xi-\eta)}\frac{\abs{p(\eta)}+s}{\abs{\xi}^2} d\xi.\]
When $\abs{\eta} > 8s$, we have $\abs{p(\eta)} \ll \abs{\eta}^2 \ll \abs{\xi}^2 + \abs{\xi-\eta}^2$, so it is easy to see that the integral is bounded in $s$. On the other hand, when $\abs{\eta} < 4s$, we certainly have $\abs{p(\eta)} \ll s^2$, and $\abs{\xi-\eta} \gg s$ in the domain of integration, and so the integral is bounded above by $s^{-N} \int_{\abs{\xi} > 100s} \frac{s^2}{\abs{\xi}^2} \jap{\xi-\eta}^{-M} d\xi \ll 1$ (by taking $M$ and $N$ to be large). We have thus bounded $\int_{\abs{\xi} >100s}$. We now turn to the second integral, where we have
\[\int_{\abs{\xi} \leq 100s} \ll \int_{\abs{\xi} \leq 100s} \abs{\phi(\xi-\eta)}\frac{\abs{p(\eta)} + s}{sd(\xi,\Sigma_\zeta)} d\xi.\]
When $\abs{\eta} > 200s$, we have $\abs{\xi-\eta} \geq s$, so the integral is bounded by $s^{-N} \int \frac{\abs{\xi}^2 + \abs{\xi-\eta}^2 + s}{sd(\xi,\Sigma_\zeta)} \jap{\xi-\eta}^{-M} d\xi \ll 1$, by~\eqref{singbound}. On the other hand, when $\abs{\eta} \leq 200s$, we have $\abs{p(\eta)} \sim sd(\eta,\Sigma_\zeta)$, and by the triangle inequality
\[\frac{\abs{p(\eta,s)} + s}{\abs{p(\xi)}} \ll \frac{\abs{\xi-\eta} + d(\xi,\Sigma_\zeta) + 1}{d(\xi,\Sigma_\zeta)} \ll 1 + d(\xi,\Sigma_\zeta)^{-1} \jap{\xi-\eta},\]
and for large $M$ our integral is bounded by
\[\int \jap{\xi-\eta}^{-M} d\xi + \int d(\xi,\Sigma_\zeta)^{-1} \jap{\xi-\eta}^{-M} d\xi.\]
The first integral is obviously finite, and the second integral is finite by~\eqref{singbound}. We have thus proven~\eqref{singularity}. The estimate~\eqref{adjointsingularity} is essentially the adjoint of this estimate, and is proven in the same way. For~\eqref{lfestimate}, we note that $\norm{u_B}_{L^2}\ll \abs{\zeta}^{-1/2} \norm{u_B}_{X_\zeta^{1/2}}$ and apply~\eqref{adjointsingularity}. Similarly, for~\eqref{hfgradestimate} we use the fact that $\abs{\xi}^2 \sim p(\xi)$ for $\abs{\xi} \geq 8s$, which implies that $\norm{\nabla(Hu_B)}_{L^2} \ll \norm{u_B}_{X_\zeta^{1/2}} \ll \norm{u}_{\dot X_\zeta^{1/2}}$. The last estimate~\eqref{hfestimate} is proven in the same way using the observation that $\abs{p(\xi)} \gg s^2$ when $\abs{\xi} \geq 8s$.
 \end{proof}

\begin{corollary}
\label{linftybilinear}
If $f$ is a bounded function and $\zeta_i \in \C^d$ are such that $\zeta_i \cdot \zeta_i = 0$ and $\abs{\zeta_1}=\abs{\zeta_2}$, then
\[\abs{\int f u_B v_B dx} \ll s^{-1}\norm{f}_{L^\infty} \norm{u}_{\dot X^{1/2}_{\zeta_1}} \norm{v}_{\dot X^{1/2}_{\zeta_2}},\]
where $\sqrt{2}s = \abs{\zeta_1}=\abs{\zeta_2}$.
\end{corollary}
\begin{proof}
We apply Cauchy-Schwarz, followed by~\eqref{lfestimate}
\begin{align*}
\abs{\int f \cdot u_B v_B dx} & \ll \abs{\int f u_B v_B dx} \\
& \ll \norm{f}_{L^\infty} \norm{u_B}_{L^2} \norm{v_B}_{L^2} \\
& \ll \abs{\zeta}^{-1}\norm{f}_{L^\infty} \norm{u}_{\dot X^{1/2}_{\zeta_1}} \norm{v}_{\dot X^{1/2}_{\zeta_2}}.
\end{align*}
\end{proof}
We now turn to the map $m_q$. We want to show that if $\gamma \in W^{1+\theta,\infty}(\R^d)$ then 
\begin{equation}
\label{qestimate}
\norm{m_q}_{\dot X^{1/2}_\zeta \to \dot X_\zeta^{-1/2}} \ll_B \omega(\norm{\nabla \log \gamma}_{W^{\theta,\infty}}) s^{-\theta},
\end{equation}
where $\omega(\epsilon) \to 0$ as $\epsilon \to 0$. We will show a bit more
\begin{theorem}
\label{bilineartheorem}
Let $\zeta_i$ satisfy $\zeta_i \cdot \zeta_i = 0$ and $\abs{\zeta_1} = \abs{\zeta_2}$. Let $\theta \in [0,1]$ be arbitrary. If $\gamma \in W^{1+\theta,\infty}(\R^d)$, then
\begin{equation}
\label{bilinearqestimate}
\abs{\jap{m_q(u),v}} \ll_B \omega(\norm{\nabla \log \gamma}_{W^{\theta,\infty}}) s^{-\theta} \norm{u}_{\dot X^{1/2}_{\zeta_1}} \norm{v}_{\dot X^{1/2}_{\zeta_2}},
\end{equation}
where $\omega(\epsilon) \to 0$ as $\epsilon \to 0$. If $\gamma \in C^1$, then
\begin{equation}
\label{endpointqestimate}
\abs{\jap{m_q(u),v}} \ll_\gamma o_{s\to \infty}(1) \norm{u}_{\dot X^{1/2}_{\zeta_1}} \norm{v}_{\dot X^{1/2}_{\zeta_2}}.
\end{equation}
\end{theorem}
 By the Leibniz rule, we have $\nabla(g^{-1} u v) = \nabla(g^{-1}) u v + g^{-1} \nabla(uv)$, so 
\begin{equation}
\label{lowerorder}
\jap{m_q(u),v} = -\int (\nabla g\cdot \nabla g^{-1}) u v dx - \int \nabla(\log g) \cdot \nabla(u v)dx.
\end{equation}
If $g$ is merely Lipschitz, then we can estimate the first term by using Corollary~\ref{linftybilinear}. For the second term, we use Lemma~\ref{localizationlemma} to prove the following estimate:
\begin{lemma}
Let $\zeta_i \in \C^d$ be such that $\zeta_i\cdot \zeta_i = 0$ and $\abs{\zeta_1}=\abs{\zeta_2}$. Let $\theta \in [0,1]$, and let $f \in W^{\theta,\infty}(\R^d)$. Then for any Schwartz functions $u,v$, we have
\begin{equation}
\label{mainestimate}
\abs{\int f \cdot \nabla(u_B v_B) dx}  \ll \norm{f}_{L^{\infty}} \norm{u}_{\dot X^{1/2}_{\zeta_1}} \norm{v}_{\dot X^{1/2}_{\zeta_2}}.
\end{equation}
If $f$ is Lipschitz, then
\begin{equation}
\label{knownestimate}
\abs{\int f \cdot \nabla(u_B v_B) dx}  \ll s^{-1}\norm{\nabla f}_{L^{\infty}} \norm{u}_{\dot X^{1/2}_{\zeta_1}} \norm{v}_{\dot X^{1/2}_{\zeta_2}}, \\
\end{equation}
so that by interpolation we also have
\begin{equation}
\label{interpolatedestimate}
\abs{\int f \cdot \nabla(u_B v_B) dx} \ll_B s^{-\theta} \norm{f}_{W^{\theta,\infty}} \norm{u}_{\dot X^{1/2}_{\zeta_1}} \norm{v}_{\dot X^{1/2}_{\zeta_2}}.
\end{equation}
\end{lemma}
\begin{proof}
To prove~\eqref{knownestimate}, we integrate by parts and apply Corollary~\ref{linftybilinear}.

If $f$ is merely bounded, then all we can say is that
\[\int \abs{f \cdot \nabla (u_B v_B)} dx \leq \norm{f}_{L^\infty} \int \abs{\nabla(u_Bv_B)} dx.\]
We can decompose the integrand using our projections onto high and low frequencies as 
\[\nabla(u_Bv_B) = \nabla(Hu_B Hv_B) + \nabla(Hu_B Lv_B) + \nabla (Lu_B Hv_B) + \nabla(Lu_B Lv_B).\]
Since the Fourier support of $Lu_B Lv_B$ is contained in a ball of radius $\ll s$, we have
\begin{align*}
\norm{\nabla(Lu_B Lv_B)}_{L^{1}} &\ll s \norm{Lu_B Lv_B}_{L^{1}} \\
&\ll s\norm{Lu_B}_{L^2} \norm{Lv_B}_{L^2}\\
& \ll s \norm{u_B}_{L^2} \norm{v_B}_{L^2},
\end{align*}
which is bounded by $\norm{u}_{\dot X^{1/2}_{\zeta_1}} \norm{v}_{\dot X^{1/2}_{\zeta_2}}$ by~\eqref{lfestimate}.

We now turn to the terms $\nabla(Hu_B Hv_B) + \nabla(Hu_B Lv_B) + \nabla(Lu_B H v_B)$. For the high-high term, we use the product rule,
\[\norm{\nabla(Hu_B Hv_B)}_{L^{1}} \ll \norm{\nabla(Hu_B)}_{L^2} \norm{Hv_B}_{L^2} + \norm{Hu_B}_{L^2} \norm{\nabla(Hv_B)}_{L^2}.\]
We combine~\eqref{hfestimate} with~\eqref{hfgradestimate} to conclude that this is also bounded by $\norm{u}_{\dot X^{1/2}_{\zeta_1}} \norm{v}_{\dot X^{1/2}_{\zeta_2}}$.

For the high-low term, we use the product rule and the finite band property
\[\norm{\nabla(Hu_B Lv_B)}_{L^{1}} \ll \norm{\nabla(Hu_B)}_{L^2} \norm{v_B}_{L^2} + s \norm{Hu_B}_{L^2} \norm{v_B}_{L^2},\]
which is bounded by $\norm{u}_{\dot X^{1/2}_{\zeta_1}} \norm{v}_{\dot X^{1/2}_{\zeta_2}}$ by~\eqref{hfgradestimate},~\eqref{lfestimate} and~\eqref{hfestimate}.

This concludes the proof of~\eqref{mainestimate}.
\end{proof}
To prove the theorem, recall that it remains to estimate the $\nabla(\log g)$ term in~\eqref{lowerorder}. The estimate~\eqref{bilinearqestimate} is straightforward, so we will prove~\eqref{endpointqestimate}. Let $\phi_\epsilon = \epsilon^{-d} \phi(x/\epsilon)$, where $\phi$ is a $C_0^\infty$ function supported on the unit ball and $\int \phi = 1$. With $f = \log g$, write $f_\epsilon = f * \phi_\epsilon$. By~\eqref{mainestimate} and~\eqref{knownestimate},
\begin{align*}
\abs{\int \nabla f \cdot \nabla(uv) dx} & \leq  \abs{\int \nabla f_\epsilon \cdot \nabla(uv) dx} + \abs{\int \nabla (f - f_\epsilon)\cdot \nabla(uv) dx} \\
& \ll (s^{-1} \norm{\nabla^2 f_\epsilon}_{L^\infty} + \norm{\nabla f - \nabla f_\epsilon}_{L^\infty}) \norm{u}_{\dot X^{1/2}_{\zeta_1}} \norm{v}_{\dot X^{1/2}_{\zeta_2}} \\
& \ll (s^{-1} \epsilon^{-1} \norm{\nabla f}_{L^\infty} + \norm{\nabla f - \nabla f_\epsilon}_{L^\infty}) \norm{u}_{\dot X^{1/2}_{\zeta_1}} \norm{v}_{\dot X^{1/2}_{\zeta_2}}.
\end{align*}
Take $\epsilon = s^{-1/2}$. Then $\epsilon \to 0$ as $s \to \infty$, so $\norm{\nabla f - \nabla f_\epsilon}_{L^{\infty}} \to 0$ by the continuity of $\nabla f$. On the other hand, $s^{-1} \epsilon^{-1} \to 0$ as well, so we have~\eqref{endpointqestimate}.
\section{An averaged estimate}
To obtain control of our solutions to the equation $(\Delta_\zeta - m_q)\phi = q$ in $\dot X^{1/2}_\zeta$, it remains to estimate $\norm{q}_{\dot X^{-1/2}_\zeta}$. The worst part of $q$ looks like $\Delta \log g = \nabla \cdot (\nabla \log g)$, and so we are led to bound expressions of the form
\[\norm{\nabla f}_{\dot X^{-1/2}}^2 := \sum_i \norm{\pd_i f}_{\dot X^{-1/2}}^2,\]
where $f$ is some continuous function with compact support. Since $f$ is compactly supported, the $\dot X^{-1/2}_\zeta$ norm is controlled by the $X^{-1/2}_\zeta$ norm. At high frequencies, $p(\xi)^{-1/2} \sim \abs{\xi}^{-1}$, so $\norm{H\nabla f}_{X^{-1/2}} \leq \norm{f}_2$. At low frequencies, however, $p(\xi)$ could be small. From the defintion of $X^{-1/2}_\zeta$ and the finite band property, we have the straightforward estimate
\[\norm{L\nabla f}_{X^{-1/2}_\zeta} \ll s^{-1/2} \norm{L\nabla f}_{L^2} \ll s^{1/2} \norm{f}_{L^2}.\]
Unfortunately, the factor $s^{1/2}$ will overpower the estimate $\norm{m_q}_{\dot X^{1/2}_\zeta \to \dot X^{-1/2}_\zeta} \ll s^{-\theta}$ that we obtained earlier unless $\theta \geq 1/2$. To overcome this problem, we will use that fact that we have at least two degrees of freedom in choosing $\zeta$. If we average over these parameters, we can obtain a better estimate that does not involve a factor of $s^{1/2}$.

Given $k \in \R^d$, we set
\begin{align*}
\zeta_1 & = s \eta_1 + i\sps{\frac k 2 + r\eta_2} \\
\zeta_2 & = -s \eta_1 + i\sps{\frac k 2 - r\eta_2}, 
\end{align*}
where $\eta_1, \eta_2 \in S^{d-1}$ satisfy $(k,\eta_1) = (k,\eta_2) = (\eta_1,\eta_2) = 0$ and $\abs{k}^2/4 + r^2 = s^2$. The vectors $\zeta_{i}$ are chosen so that $\zeta_{i} \cdot \zeta_{i} = 0$ and $\zeta_1 + \zeta_2 = ik$. Our goal is to find a sequence $s^{(n)}$, $\eta_i^{(n)}$ such that $s^{(n)} \to \infty$ and $\norm{q}_{\dot X^{-1/2}_{\zeta}} \to 0$ for $\zeta = \zeta_i^{(n)}$, $i=1,2$. Let $P$ be some fixed two-plane through the origin normal to $k$. Then $\eta_1$ can be taken to be any vector in $P \cap S^{d-1}$, which we identify with $S^1$. We choose $\eta_2 \in P \cap S^{d-1}$ orthogonal to $\eta_1$. For definiteness, we can require $\{\eta_1,\eta_2\}$ to be positively oriented, so that $\eta_2$ is determined by the choice of $\eta_1$, but this will not be relevant to our calculations. The idea now is that by averaging over all possible choices of $\eta_1\in S^1$ and $s$ in a dyadic region $[\lambda,2\lambda]$, we can get an improved estimate for the size of the Sch\"odinger potential in $\dot X_{\zeta_i}^{-1/2}$. 
\begin{lemma}
\label{averageestimate}
Let $\theta \in [0,1]$. Let $\phi_B$ be a smooth function with compact support. Fix $k \in \R^d$ and a two-plane $P \perp k$, and let $\zeta_{i} = \zeta_{i}(s,\eta_1)$ be as above. Then for sufficiently large $\lambda$,
\begin{equation}
\label{avgestimateinterpolated}
\int_{S^1}\int_{\lambda}^{2\lambda} \norm{\phi_B \nabla f}_{\dot X^{-1/2}_{\zeta_i}}^2 dsd\eta_1 \ll_{\phi_B,k} \lambda^{1-\theta} \norm{f}_{H^\theta}^2.
\end{equation}
If $f \in L^2$, then
\begin{equation}
\label{averageendpointestimate}
\int_{S^1}\int_{\lambda}^{2\lambda} \norm{\phi_B \nabla f}_{\dot X^{-1/2}_{\zeta_i}}^2 dsd\eta_1 \ll_{\phi_B, k,f} o_{\lambda \to \infty}(\lambda).
\end{equation}
\end{lemma}
\begin{proof}
We will prove the estimate~\eqref{avgestimateinterpolated} by interpolating the estimates
\begin{align}
\label{avgestimatederivative}
\int_{S^1}\int_{\lambda}^{2\lambda} \norm{\phi_B \nabla f}_{\dot X^{-1/2}_{\zeta_i}}^2 dsd\eta_1  & \ll \norm{\nabla f}_{L^2}^2\\
\label{avgestimate}
\int_{S^1}\int_{\lambda}^{2\lambda} \norm{\phi_B \nabla f}_{\dot X^{-1/2}_{\zeta_i}}^2 dsd\eta_1 & \ll \lambda \norm{f}_{L^2}^2.
\end{align}
For both of these estimates, we use the fact that 
\[\norm{\phi_B \nabla f}_{\dot X_{{\zeta_i}_i}^{-1/2}} \ll \norm{\nabla f}_{X_{\zeta_i}^{-1/2}}.\]
The first estimate~\eqref{avgestimatederivative} is then immediate since 
\[\norm{\nabla f}_{X_{\zeta_i}^{-1/2}}^2 \ll s^{-1} \norm{\nabla f}_{L^2}^2 \ll \lambda^{-1} \norm{\nabla f}_{L^2}^2\]
for $s \in [\lambda,2\lambda]$.

We now proceed to the second estimate~\eqref{avgestimate}. By definition,
\[\norm{\nabla f}_{X_{\zeta_i}^{-1/2}}^2 = \int \frac{\abs{\xi}^2}{\abs{p_{\zeta_i}(\xi)} + s} \abs{\hat f(\xi)}^2 d\xi.\]
We want to show that
\[\int_{S^1} \frac{1}{\lambda} \int_{\lambda}^{2\lambda} \int_{\R^d} \frac{\abs{\xi}^2}{\abs{p_{\zeta_i}(\xi)} + s} \abs{\hat f(\xi)}^2 d\xi \ll \int_{\R^d} \abs{\hat f(\xi)}^2 d\xi.\]
When $\abs{\xi} \ggg s$, we have $\abs{p(\xi)} \sim \abs{\xi}^2$, so this part of the integral gives us no trouble. The problem is that $p(\xi)$ does not have to be comparable to $\abs{\xi}^2$ when $\abs{\xi}$ is small compared with $s$. In fact, when $\xi \in \Sigma_{\zeta_i}$ (where $\Sigma_{\zeta_i} = \{\xi: p_{{\zeta_i}}(\xi)=0\}$), the best bound for the integrand is $s \abs{\hat f(\xi)}^2$. 

To get around this, we use the fact that any given $\xi$ cannot be contained in all of the codimension-2 spheres $\Sigma_{\zeta_i}$ at once. That is, if we fix a point $\xi$, then by varying the size of the sphere (which acts roughly like $s$) and the hyperplane in which it lies (which depends on $\eta_1$), we will mostly miss $\xi$. In fact, on average the denominator $\abs{p(\xi)} + s$ will be comparable to $\abs{\xi}^2$. 

More precisely, write 
\[p(\xi) \sim \abs{-\abs{\xi}^2 \pm 2r\xi\cdot \eta_2 - k\cdot \xi} + 2s\abs{\xi\cdot \eta_1}.\]
Assume $\lambda \geq 100 \jap{k}$. When $\abs{\xi} \ggg \abs{k}$, the $k \cdot \xi$ term is much smaller than the other terms, which are potentially of size $\abs{\xi}^2$. On the other hand, there is no problem in the region $\abs{\xi} \ll \abs{k}$, since the estimates are allowed to depend on $k$. We can thus ignore the $k \cdot \xi$ term by focusing our attention on the complement of the set $F : = \{\xi: \abs{\xi} \leq 100\jap{k}\}$.

Let $\xi_i = \eta_i \cdot \xi$. There are two ways we will get $p(\xi)$ to be comparable to $\abs{\xi}^2$. One way is to exploit the term $-\abs{\xi}^2$, but this only works if $2s\abs{\xi_2} \lll \abs{\xi}^2$. The other way is to use the terms $2r\xi_2$ and $2s\xi_1$, but these can only be comparable to $\abs{\xi}^2$ if the component of $\xi$ lying in $P$ is $\gg \abs{\xi}^2/s$. Thus we set $E^{(s)} = \{\xi: \abs{\xi^\perp} \geq \abs{\xi}^2/100s\}$, where $\xi^\perp$ is the orthogonal projection of $\xi$ onto $P$. 

Write $f = b^{(s)} + g^{(s)} + h$, where 
\begin{align*}
\hat h & = \chi_F \hat f\\
\hat b^{(s)} & = (1-\chi_F) \chi_{E^{(s)}} \hat f \\
\hat g^{(s)} & = (1-\chi_F)(1-\chi_{E^{(s)}}) \hat f.
\end{align*}
We estimate each term separately. First, 
\[\norm{\nabla h}_{X^{-1/2}_{\zeta_i}}^2 \ll s^{-1}\norm{f}_{L^2}^2 \ll \lambda^{-1}\norm{f}_{L^2}^2,\]
so~\eqref{avgestimate} holds for this part. 

For $\xi \notin E^{(s)} \cup F$, we have 
\[\abs{\xi_2} \leq \abs{\xi^\perp} \leq \abs{\xi}^2/100s, \quad\quad \abs{k} \leq \abs{\xi}/100.\]
 We then have 
\begin{align*}
\abs{\pm 2r\xi_2 - k\cdot \xi} &\leq \abs{\xi}^2/2\\
\abs{-\abs{\xi}^2 \pm 2r\xi_2  - k \cdot \xi} &\geq \abs{\xi}^2/2,
\end{align*}
so 
\[\norm{\nabla g^{(s)}}_{X^{-1/2}_{\zeta_{i}}} \ll \norm{f}_{L^2}.\]

We now turn to $b^{(s)}$. Note that if $\xi \in E^{(s)}$, we have $\abs{\xi}^2 \ll s \abs{\xi^\perp} \leq s \abs{\xi}$, so that $\abs{\xi} \ll s$ in this region.

Let $w(s,\eta_1,\xi) := \abs{\xi}^2/(\abs{p_{\zeta_i}(\xi)} + s)$. We wish to show that 
\begin{equation}
\label{integralestimate}
\int_{\eta_1 \in S^1} \int_\lambda^{2\lambda} w(s,\eta_1,\xi) ds d\eta_1 \ll \lambda.
\end{equation}
Let $\mu$ be some positive parameter. We first establish
\begin{align}
\notag
\int_\lambda^{2\lambda} \frac{\abs{\xi}^2}{\abs{-\abs{\xi}^2 \pm 2r\xi_2  - k \cdot \xi} + \mu \lambda}ds& \ll \int_{\lambda/2}^{2\lambda} \frac{\abs{\xi}^2}{\abs{-\abs{\xi}^2 \pm 2r\xi_2 - k\cdot \xi}+\mu \lambda} dr\\
& \ll \frac{\abs{\xi}^2}{\abs{\xi_2}} \log \sps{\frac{\abs{\xi_2}}{\mu} + 1}.
\label{logintegral}
\end{align}
Here we have used the fact that $s^2 = r^2 + \abs{k}^2/4$, so that $2s ds = 2r dr$ and $ds \ll dr$. We certainly have $r \leq 2\lambda$ when $s \leq 2\lambda$, and $s \geq \lambda$ implies $r \geq \lambda/2$, since $k \lll \lambda$. The last line follows from the following calculation: If $c,\lambda > 0$, then
\begin{align*}
\int_I \frac{1}{\abs{ax+b} + c} dx & = \frac{1}{\abs{a}}\int_{I}  \frac{1}{\abs{x + (b/a)} + (c/\abs{a})} dx \\
& \ll \frac{1}{\abs{a}} \int_0^{\abs{I}} \frac{1}{x + (c/\abs{a})} dx \\
& \ll \frac{1}{\abs{a}} \log\sps{\frac{\abs{I}+(c/\abs{a})}{c/\abs{a}}} \\
& \ll \frac{1}{\abs{a}} \log\sps{\frac{\abs{a}\abs{I}}{c} + 1}.
\end{align*}
Applying this with $a = \pm 2\xi_2$, $b = -\abs{\xi}^2-k\cdot \xi$, $c = \mu \lambda$, we have~\eqref{logintegral}. 

Now, for any $\mu > 0$, the portion of $S^1$ on which $\abs{\eta_1 \cdot \xi} \leq \mu$ (recall that $\eta_1 \in S^1 \subset P$) has length $\ll \mu/\abs{\xi^\perp}$. If $\abs{\eta_1\cdot \xi} \in [\mu/2,\mu]$, we have
\[w(s,\eta_1,\xi) \ll \frac{\abs{\xi}^2}{\abs{-\abs{\xi}^2 \pm 2r\xi_2  - k \cdot \xi} + \mu \lambda},\]
so by~\eqref{logintegral} it follows that
\begin{align*}
\int_{\abs{\eta_1\cdot \xi} \in [\mu/2,\mu]} \int_{\lambda}^{2\lambda} w(s,\eta_1,\xi) ds d\eta_1 &\ll \frac{\mu}{\abs{\xi^\perp}} \frac{\abs{\xi}^2}{\abs{\xi_2}} \log\sps{\frac{\abs{\xi_2}}{\mu} + 1} \\
& \ll \sps{\frac{\mu}{\abs{\xi_2}}}^{1-\epsilon} \frac{\abs{\xi}^2}{\abs{\xi^{\perp}}}\\
& \ll \sps{\frac{\mu}{\abs{\xi_2}}}^{1-\epsilon} \lambda
\end{align*}
for any $\epsilon \in (0,1)$. If $\abs{\eta_1\cdot \xi} \leq \abs{\xi^{\perp}}/2$, then $\abs{\eta_2 \cdot \xi} \geq \abs{\xi^{\perp}}/2$, so when $\mu \leq \abs{\xi^{\perp}}/4$, we have can replace $\abs{\xi_2}$ in the above estimate by $\abs{\xi^{\perp}}$. Setting $\epsilon = 1/2$, $\mu = 2^{-1-j} \abs{\xi^{\perp}}$ and summing over $j$,
\begin{equation}
\label{smallintervals}
\sum_{j \geq 0} \int_{\abs{\eta_1\cdot \xi} \in [2^{-2-j}, 2^{-1-j}]\abs{\xi^{\perp}}} \int_{\lambda}^{2\lambda} w(s,\eta_1,\xi)ds d\eta_1 \ll \sum_{j \geq 0} 2^{-j/2} \lambda \ll \lambda.
\end{equation}
On the other hand, when $\abs{\eta_1\cdot \xi} \geq \abs{\xi^{\perp}}/2$, we have $w(s,\eta_1,\xi) \ll \abs{\xi}^2/s\abs{\xi^{\perp}} \ll 1$, so
\begin{equation}
\label{largeintervals}
\int_{\abs{\eta_1\cdot \xi} \geq \abs{\xi^{\perp}}/2} \int_{\lambda}^{2\lambda} w(s,\eta_1,\xi) \ll \lambda.
\end{equation}
Combining~\eqref{smallintervals} and~\eqref{largeintervals}, we get~\eqref{integralestimate}. By Fubini's theorem, this implies that
\[\int_{\eta\in S^1} \int_{\lambda}^{2\lambda} \norm{\nabla b^{(s)}}_{X^{-1/2}_{\zeta_{i}}}^2 ds \ll \lambda \norm{f}_{L^2}^2.\]
Combining this with the estimates for $h$ and $g^{(s)}$, we have
\[\int_{\eta \in S^1} \int_{\lambda}^{2\lambda} \norm{\nabla f}_{X^{-1/2}_{\zeta_{i}}}^2ds \ll \lambda \norm{f}_{L^2}^2,\]
which concludes the proof of~\eqref{avgestimate}.

For~\eqref{averageendpointestimate}, we argue as in Theorem~\ref{bilineartheorem}. That is, we write $f = f_\epsilon + (f-f_{\epsilon})$, and obtain
\[\int_{\eta \in S^1} \int_{\lambda}^{2\lambda} \norm{\nabla f}_{X^{-1/2}_{\zeta_{i}}}^2ds \ll (\lambda^{-1} \norm{\nabla f_\epsilon}_{L^2}^2 + \norm{f-f_\epsilon}_{L^2}^2)\lambda = o_{\lambda\to \infty}(\lambda)\]
by taking $\epsilon = \lambda^{-1/4}$.
\end{proof}
\section{CGO solutions}
We now use these estimates to solve the equation $\Delta_\zeta \psi =q(1+\psi) $, which we can write as
\begin{equation}
\label{fixedpoint}
\psi = \Delta_\zeta^{-1} q \psi + \Delta_\zeta^{-1} q.
\end{equation}
\begin{theorem}
Let $\gamma_i$ ($i = 1,2$) be $W^{1,\infty}(\R^d)$ functions with $\gamma_i > c > 0$ and $\gamma_i = 1$ outside some ball $B$. Let $q_i = \Delta \gamma_i^{1/2}/\gamma_i^{1/2}$ be the corresponding potentials. Then for fixed $k$ there exists a sequence of pairs $\zeta_1^{(n)}, \zeta_2^{(n)}$ with $s^{(n)} \to \infty$ such that for each $\zeta = \zeta^{(n)}_{j}$ the potentials $q_i$ satisfy
\begin{equation}
\label{qdecay}
\norm{q_i}_{\dot X^{-1/2}_{\zeta_j}} \to 0
\end{equation}
as $s^{(n)} \to \infty$. If $\norm{\nabla \log \gamma}_{L^\infty}$ is sufficiently small, or if $\gamma \in C^1$, then there are solutions $u_i = e^{x \cdot \zeta_i} (1 + \psi_i)$ to the equation $(-\Delta + q_i) u_i = 0$, and
\begin{equation}
\label{solutiondecayendpoint}
\norm{\psi_i}_{\dot X^{1/2}_{\zeta_i}} \to 0
\end{equation}
as $s^{(n)} \to \infty$.
\end{theorem}
\begin{proof}
By Lemma~\ref{averageestimate} and~\eqref{lowerorder}, we have
\[\sum_{i,j} \frac 1 \lambda \int_{\eta_1 \in S^1}\int_\lambda^{2\lambda} \norm{q_i}_{\dot X^{-1/2}_{\zeta_j}}ds \to 0\]
as $\lambda \to \infty$. Thus there exist $\zeta = \zeta_{i}$ with $s \to \infty$ such that $\sum_{i,j} \norm{q_i}_{\dot X_{\zeta_j}^{-1/2}} \to 0$. This implies that for $q=q_i$ and $\zeta=\zeta_j$, we have $\norm{\Delta_{\zeta}^{-1} q}_{\dot X^{1/2}_\zeta} \to 0$. By~\eqref{qestimate} and~\eqref{endpointqestimate}, we can make $\norm{\Delta_\zeta^{-1} m_q}_{\dot X^{1/2}_\zeta \to \dot X^{1/2}_\zeta} = \norm{m_q}_{\dot X^{1/2}_\zeta \to \dot X_\zeta^{-1/2}}$ less than one. If $\gamma \in C^1$, we do this using~\eqref{endpointqestimate}, while if $\gamma \in W^{1,\infty}$ we do this using~\eqref{qestimate} and the fact that $\norm{\log \gamma}_{W^{1,\infty}}$ is small. Thus by the contraction mapping principle, there exists a solution to~\eqref{fixedpoint} satisfying $\norm{\psi}_{\dot X^{1/2}_\zeta} \to 0$ 
\end{proof}

\section{Proof of uniqueness}
To prove our main theorem, we first reduce to the case where the $\gamma_i$ are globally defined.
\begin{lemma}
\label{extensionlemma}
Let $\gamma_1,\gamma_2 \in W^{1,\infty}(\bar \Omega)$ be Lipschitz (or $C^1$) conductivities. If $\Lambda_{\gamma_1} = \Lambda_{\gamma_2}$, then each $\gamma_i$ can be extended to a function in $W^{1,\infty}(\R^d)$ (or $C^1(\R^d)$) with $\gamma_i = 1$ outside some ball $B$ satisfying 
\[\jap{m_{q_1} v_1,v_2} = \jap{m_{q_2} v_1, v_2}\]
for any solutions $v_i \in H^1_{\loc}(\R^d)$ to the Schr\"odinger equation $(-\Delta + q_i)v_i = 0$. 
\end{lemma}
\begin{proof}
This is essentially due to~\cite{Sylvester1987}, and in this case follows from the boundary determination result of~\cite{Alessandrini1990}. The details for the case of Lipschitz conductivities are worked out in the proof of Theorem 0.7 of~\cite{Brown1996}. Note that the boundary determination result is only needed to identify $\gamma_1$ and $\gamma_2$ at the boundary, for which the Lipschitz hypothesis is sufficient.
\end{proof}
The following argument is apparently due to Alessandrini~\cite{Brown1996}:
\begin{lemma}
\label{gaidentity}
Let $\gamma_i \in W^{1,\infty}(\R^d)$ and $g_i = \gamma_i^{1/2}$ be as above, and suppose $g_1^{-1} \Delta g_1= g_2 \Delta g_2^{-1}$ in the sense that
\[\int \nabla g_1 \cdot \nabla(g_1^{-1} \phi) dx = \int \nabla g_1 \cdot \nabla(g_1^{-1} \phi) dx\]
whenever $\phi \in C_0^1(\R^d)$. Then $g_1 = g_2$.
\end{lemma}
\begin{proof}
For $\phi = g_1 g_2 \psi$, we have
\[\int g_1g_2 \nabla(\log g_1 - \log g_2) \cdot \nabla \psi = 0.\]
Set $\psi = \log g_1 - \log g_2$. Since the $g_i$ are nonnegative and agree outside some large ball, this implies $g_1=g_2$.
\end{proof}
\begin{remark*}
In order to control the $\norm{\nabla \log \gamma_i}_{L^\infty(\R^n)}$ norm of the extensions, we must in principle assume that the $\norm{\log \gamma_i}_{W^{1,\infty}(\Omega)}$ are small. In this case, however, we may dispense with the hypothesis that $\norm{\log \gamma_i}_{L^\infty(\R^n)}$ is small\footnote{We would like to thank one of the anonymous reviewers for pointing out that this relaxed hypothesis should be sufficient.}. This is because solving the equation $L_\gamma u = 0$ is equivalent to solving the equation $L_{c\gamma} u = 0$, where $c$ is any positive constant. The quantities $\norm{\nabla \log \gamma_i}_{L^\infty}$ do not change when we replace $\gamma_i$ by $c_i\gamma_i$. By choosing $c_i$ appropriately, we can ensure that the average $\fint_\Omega \log (c_i\gamma_i(x)) dx$ vanishes, and since $\Omega$ is bounded and connected we have
\[\norm{\log c_i \gamma_i}_{L^\infty(\Omega)} \ll \norm{\nabla \log \gamma_i}_{L^\infty(\Omega)}.\]
In other words, if $\norm{\nabla \log \gamma_i}_{L^\infty(\Omega)}$ is small enough, then $\norm{\log c_i \gamma_i}_{W^{1,\infty}(\Omega)}$ will be small, and we can construct CGO solutions to $L_{\gamma_i} u = 0$.
\end{remark*}
We now use a bootstrapping argument to get $v \in H^1_{\loc}(\R^d)$ (cf.~\cite{Brown1996}). If we set $v = e^{x\cdot \zeta}(1 + \psi)$, then $v \in L^2_{\loc}$ by~\eqref{lfestimate}. By construction, we also have $\Delta v = qv$. Now, $q(1 + \psi) \in \dot X_\zeta^{-1/2}$, by~\eqref{qdecay} and~\eqref{qestimate}. Since $\abs{p_\zeta(\xi)} \ll_{\zeta} \jap{\xi}^2$, this implies that $q(1+\psi) \in H^{-1}$. Since multiplying by a smooth compactly supported function preserves $H^{-1}$, we have $\Delta v \in H^{-1}$, which implies in turn that $v \in \dot H^{1}$. Since $v$ is already in $L^2_{\loc}$, we have $v \in H^1_{\loc}$ as well.

If $\gamma_1$ and $\gamma_2$ are two Lipschitz conductivities, then by Lemma~\ref{extensionlemma} we have $\jap{m_{q_1}v_1,v_2} = \jap{m_{q_2}v_1,v_2}$ for the CGO solutions $v_1,v_2$ that we have constructed. Fix $k\in\R^d$, and let $v_i$ be corresponding CGO solutions corresponding to $\zeta_i$ as constructed in the theorem. Let $\phi\in C_0^\infty$ be a cutoff function such that $\phi = 1$ on the supports of the $q_i$. Then we have
\begin{align*}
\jap{m_{q_i}v_1,v_2} &= \jap{m_{q_i}\phi v_1,\phi v_2}\\
&=\jap{m_{q_i}\phi e^{x\cdot \zeta_1}(1 + \psi_1),\phi e^{x\cdot \zeta_2}(1+\psi_2) } \\
&=\jap{q_i, e^{ix\cdot k}} + \jap{m_{q_i}\phi e^{ix\cdot k}, \psi_1 + \psi_2} + \jap{m_{q_i} \phi e^{ixk/2} \psi_1, \phi e^{ixk/2} \psi_2}.
\end{align*}
Now $\tilde \phi := \phi e^{ix\cdot k/2} \in C_0^\infty$, and we have
\begin{align*}
\abs{\jap{m_{q_i}\phi e^{ix\cdot k}, \psi_1 + \psi_2}} & \ll \norm{\tilde \phi q_i}_{\dot X^{-1/2}_{\zeta_1}} \norm{\psi_1}_{\dot X^{1/2}_{\zeta_1}}+\norm{\tilde \phi q_i}_{\dot X^{-1/2}_{\zeta_2}} \norm{\psi_2}_{\dot X^{1/2}_{\zeta_2}}\\
 & \ll \norm{q_i}_{X^{-1/2}_{\zeta_1}} \norm{\psi_1}_{\dot X^{1/2}_{\zeta_1}}+\norm{q_i}_{X^{-1/2}_{\zeta_2}} \norm{\psi_2}_{\dot X^{1/2}_{\zeta_2}}\\
& \to 0
\end{align*}
as $s \to \infty$ by~\eqref{solutiondecayendpoint} and~\eqref{qdecay}. Finally, we have
\[\abs{\jap{m_{q_i} \phi e^{ixk/2} \psi_1, \phi e^{ixk/2} \psi_2}} \ll \norm{\psi_1}_{\dot X^{1/2}_{\zeta_1}} \norm{\psi_2}_{\dot X^{1/2}_{\zeta_2}} \to 0\]
by~\eqref{bilinearqestimate},~\eqref{adjointsingularity} and~\eqref{solutiondecayendpoint}. Thus for any $k \in \R^d$ we have
\[\int \nabla \gamma_1^{1/2} \cdot \nabla(\gamma_1^{-1/2}e^{ix\cdot k}) dx = \int \nabla \gamma_2^{1/2} \cdot \nabla(\gamma_2^{-1/2}e^{ix\cdot k}) dx.\]
This implies $q_1 = q_2$ in the sense of Lemma~\ref{gaidentity}.
\bibliography{/home/boaz/Documents/library.bib}{}
\bibliographystyle{amsalpha}
\end{document}